\documentclass[preprint]{elsarticle}
\usepackage{lineno,hyperref}
\usepackage{amsmath}
\usepackage{amsthm}
\usepackage{amssymb}
\usepackage{color}
\usepackage{comment}

\usepackage{graphicx}
\usepackage{enumitem}
\usepackage{comment}

\newtheorem{theorem}{Theorem}[section]
\newtheorem{remark}[theorem]{Remark}
\newtheorem{proposition}[theorem]{Proposition}

\newtheorem{corollary}[theorem]{Corollary}

\begin{document}

\begin{frontmatter}

\title{On graphs with eigenvectors in $\{1, -1, 0\}$ and the max $k$-cut problem}

\author[Jorgeaddress]{Jorge Alencar}
\ead{jorgealencar@iftm.edu.br}
\address[Jorgeaddress]{Instituto Federal de Educa\c{c}\~{a}o, Ci\^{e}ncia e Tecnologia do Tri\^angulo Mineiro, Brasil}

\author[Leoaddress1]{Leonardo de Lima}
 \ead{leonardo.delima@ufpr.br}
\address[Leoaddress1]{Departamento de Administração Geral e Aplicada, Universidade Federal do Paran\'a, Brasil}

\author[Vladoaddress]{Vladimir Nikiforov}
 \ead{vnikifrv@memphis.edu}
 \address[Vladoaddress]{Department of Mathematical Sciences, University of Memphis, Memphis, TN 38152, USA}

\begin{abstract}
In this paper, we characterize all graphs with eigenvectors of the signless Laplacian and adjacency matrices with components equal to $\{- 1, 0, 1\}.$ We extend the graph parameter max $k$-cut to square matrices and prove a general sharp upper bound, which implies upper bounds on the max $k$-cut of a graph using the smallest signless Laplacian eigenvalue, the smallest adjacency eigenvalue, and the largest Laplacian eigenvalue of the graph. In addition, we construct infinite families of extremal graphs for the obtained upper bounds.
\end{abstract}


 \end{frontmatter}

\section{Introduction}\label{intro}

Throughout this paper, we consider $G=(V,E)$ a graph with \textcolor{blue}{vertex set $V=\{v_{1}, v_{2}, \ldots, v_{n}\}$} and edge set $E$ such that  $|V|=n$ and $|E|=m.$ 
We write $A(G)=A$ for the adjacency matrix of $G,$ where $a_{ij}= 1$ if $e_{ij} \in E$ and $a_{ij}=0$ otherwise. The diagonal matrix $D(G)$ is given by the row-sums of $A,$ i.e., the degrees of $G.$ As usual, $L(G)=D(G) - A(G)$, denotes the Laplacian matrix of $G$, and $Q(G)=D(G)+A(G)$, denotes the signless Laplacian matrix of $G.$

In \cite{wilf}, Wilf posed the following question ``Which graphs have eigenvectors with entries solely $\pm 1$?''. Some recent papers have studied this subject in different matrices associated with a graph. In \cite{dragan2016}, Stevanovic showed that the problem of finding such graphs is NP-Hard. In \cite{caputo2019}, Caputo, Khames, and A. Knippel described all graphs whose Laplacian matrix has eigenvectors with entries only $\pm 1.$ Alencar and de Lima, in Theorem 4.1 of \cite{AL21}, solved the question raised by Wilf by presenting necessary and sufficient conditions to a graph $G$ has an eigenvector with only entries $\{-1, 1\}.$ In addition, Caputo,  Khames, and Knippel  \cite{caputo2019} characterized all graphs with Laplacian eigenvectors entries equal to $-1$, 0, or 1. 

In this paper, we give a complete characterization of all graphs with an eigenvector to the signless Laplacian and the adjacency matrix that has entries in $\{-1, 0, 1\}$. Our results point out an interesting relation of the graphs with this kind of eigenvectors to the degree sequence of the graph. We also characterize all graphs with eigenvectors of the signless Laplacian and adjacency matrix with entries in $\{c_1, -c_2, 0\}$ for any positive real numbers $c_1$ and $c_2.$

The \emph{maximum }$k$\emph{-cut} of $G,$ denoted by $\mathrm{mc}_{k}\left(
G\right)  $, is the maximum number of edges in a $k$-partite subgraph of $G.$ In \cite{vanDam2016}, van
Dam and Sotirov showed that
\begin{equation}
\mathrm{mc}_{k}\left(  G\right)  \leq\frac{n\left(  k-1\right)  }{2k}%
\mu_1\left(  G\right)  , \label{DS}%
\end{equation}
where $\mu_{1}\left(  G\right)  $ is the maximum eigenvalue of the Laplacian
matrix of $G.$ 

Somewhat later, Nikiforov \cite{niki2016} showed that if $G$ is a graph with $n$ vertices and $m$ edges, then
\begin{equation}
\mathrm{mc}_{k}\left(  G\right)  \leq\frac{k-1}{k}\left(  m-\frac{\lambda_{n
}\left(  G\right)  n}{2}\right)  . \label{mcin}%
\end{equation}
Here $\lambda_{n}\left(  G\right)$ stands for the smallest eigenvalue of the adjacency matrix.

In the present paper we extend the concept of max $k$-cut to square matrices and derive a general upper bound, which implies the two bounds above, and in addition yields a novel one using the smallest eigenvalue of the signless Laplacian.
 
The paper is organized as follows. In Section \ref{sec:basics} we present some notation and preliminary results about graphs with  eigenvectors with entries only in $\{c_1, -c_2, 0\},$ where $c_1, c_2$ are positive real numbers. In Section 3, we study bounds on the $k$-cut of square matrices and graphs. Section 4 presents graphs with a specific type of eigenvectors, which is useful to build the infinite families of extremal graphs of Section 5.

\section{Preliminaries}\label{sec:basics}

Let us introduce some notation:  $e_{n}$ is the all ones $n-$vector; $A_{ij} = A(S_i,S_j)$ is the adjacency matrix of the induced subgraph $G[S_{i} \cup S_{j}]$; $E(X,Y)$ is the set of edges with endpoints in the vertices sets $X$ and $Y;$ $G[E(X,Y)]$ is the edge-induced subgraph by the edges in $E(X,Y);$ 
let $d_{v_i}(G)$ be the degree of vertex $v_{i} \in V;$   $d_{v_i}(X,\overline{X})$ is the number of edges connecting vertex $v_i \in X$ to vertices of $\overline{X};$ $d_{v_i}(X)$  is the degree of vertex $v_i$ in the induced subgraph $G[X].$

Suppose that $V$ is the vertex set of a graph
$G$ and let $V=S_{1}\cup S_{2}\cup S_{3}$ be a tripartition of $V.$ A vector
$p$ indexed by $V$ is said to be associated with the tripartition $\left\{
S_{1},S_{2},S_{3}\right\}  $ if the entries of $p$ are equal within each of
the sets $S_{1},S_{2},S_{3}.$
%
We define $\,p\,$ as an $n$-vector as follows:
$$p_{i} = \left\lbrace \begin{array}{rl}
c_1,&\text{ for }v_i \in S_1,\\
c_2,&\text{ for }v_i \in S_2,\\
c_3 ,&\text{ for }v_i \in S_3.
\end{array}\right. $$
When $p\in \{c_1, c_2\}^n$, we proceed analogously with a bipartition instead of a tripartition.
In particular, we are interested in the case where
$c_1 = 1, c_2 = -1,$ and $c_3 = 0.$
In this section, we use the proof technique of \cite{AL21} to characterize all graphs with eigenvectors with entries in $\{1,-1,0\}.$ Similar result is also obtained for the eigenvectors of the Laplacian matrix and the signless Laplacian matrix.

Next, we state necessary and sufficient conditions for a vector associated with a tripartition to be an eigenvector of the adjacency matrix of a graph.

\begin{theorem}\label{thm:adjpartition} Let $G=(V,E)$ be a graph on $n$ vertices and let $p\in \{1,-1,0\}^n$  be the vector associated to the partition $\{S_1, S_2, S_3\}$ of $V.$ Then, $p$ is an eigenvector of $A$ to the eigenvalue $\lambda$ if and only if
\begin{eqnarray}
\label{eq:adjpartition}
\begin{array}{cccl}
d_{v_i}(S_1) -  d_{v_i}(S_1,S_2)&=& \lambda,& \mbox{for  } \; v_i \in S_1,\\
d_{v_i}(S_2)- d_{v_i}(S_2,S_1)&=&\lambda,& \mbox{for 
 }\; v_i \in S_2,\\
d_{v_i}(S_3,S_1)- d_{v_i}(S_3,S_2)&=&0,& \mbox{for  }\; v_i \in S_3.\end{array}
\end{eqnarray}
\end{theorem}
\begin{proof}
Let $A(G)=A$ be the adjacency matrix  of $G$ and let $p = (e_{n_1} \;  -e_{n_2} \; 0_{n_3})^{T}$ be the vector associated to the tripartition $\{S_1, S_2, S_3\}$ of $V$, where $n_j = |S_j|$, for $j=1,2,3.$ Thus, we have 
$$A p  = \left(
	 \begin{array}{ccc}
	 A_{11} &A_{12}&A_{13}\\
	 A_{21} &A_{22}&A_{23}\\
	 A_{31} &A_{32}&A_{33}\\
	 \end{array}\right)
	 \left(
	 \begin{array}{r}
	e_{n_1}\\
	 -e_{n_2}\\
	 0_{n_3}
	 \end{array}\right)	    	
	  \nonumber $$
	 $$ = \left(
	 \begin{array}{c}
	 A_{11}e_{n_1}-A_{12}e_{n_2}+A_{13}0_{n_3}\\
	 A_{21}e_{n_1}- A_{22}e_{n_2}+A_{23}0_{n_3}\\
	 A_{31}e_{n_1}- A_{32}e_{n_2}+A_{33}0_{n_3}\\	
	 \end{array}\right)	\nonumber $$
	 $$= \left(
	 \begin{array}{c}
	 d_{v_1}(S_1) - d_{v_1}(S_1,S_2) \\
	 \vdots \\
	 d_{v_{n_1}}(S_1) - d_{v_{n_1}}(S_1,S_2)\\
	 d_{v_{n_1+1}}(S_2,S_1)-  d_{v_{n_1+1}}(S_2)\\
	 \vdots \\
	 d_{v_{n_1+n_2}}(S_2,S_1)-  d_{v_{n_1+n_2}}(S_2)\\
	 d_{v_{n_1+n_2+1}}(S_3, S_1) -  d_{v_{n_1+n_2+1}}(S_3,S_2)\\
	 \vdots \\
	 d_{v_{n_1+n_2+n_3}}(S_3, S_1) -  d_{v_{n_1+n_2+n_3}}(S_3,S_2)\\
	 \end{array}\right).	\nonumber$$

Then, from the eigenequation $A p = \lambda p$ and using the previous equations, we get
\begin{eqnarray}
d_{v_i}(S_1)-d_{v_i}(S_1,S_2)&=& \lambda, \mbox{  for} \; v_i \in S_1,\\
d_{v_i}(S_2)-d_{v_i}(S_2,S_1)&=& \lambda, \mbox{  for}\; v_i \in S_2,\\
d_{v_i}(S_3,S_1)-d_{v_i}(S_3,S_2)&=&0, \mbox{  for}\; v_i \in S_3. \label{eq:adj3}
\end{eqnarray}
The proof is complete.
\end{proof}

\begin{remark}\label{rem1}
    It is easy to see that Theorem \ref{thm:adjpartition} holds for weighted graphs with possible loops. This fact will be used to prove Theorem \ref{thm:slpartition}.
\end{remark} 

Notice that if the induced subgraphs $G[S_1]$ and $G[S_2]$ are isomorphic $r$-regular graphs and $G[E(S_1, S_2)]$ is a $s$-regular bipartite graph, then the graph $G$ has at least one eigenvector with entries in $\{1, -1, 0\}$ associated with the eigenvalue $r-s$ for any $G[S_3]$ satisfying condition \eqref{eq:adj3}. For instance, Figure \ref{fig:family1} displays an example of a graph satisfying Theorem \ref{thm:adjpartition} with $\lambda_{10} = -3,$ where the bold dashed-line represents a join operation between vertices sets $S_1$ and $S_2.$ Also, once Equation \eqref{eq:adj3} is satisfied for any vertex $v \in S_3$, the induced subgraph $G[S_3]$ may be any graph.
\begin{figure}
\centering
\includegraphics[height=6cm]{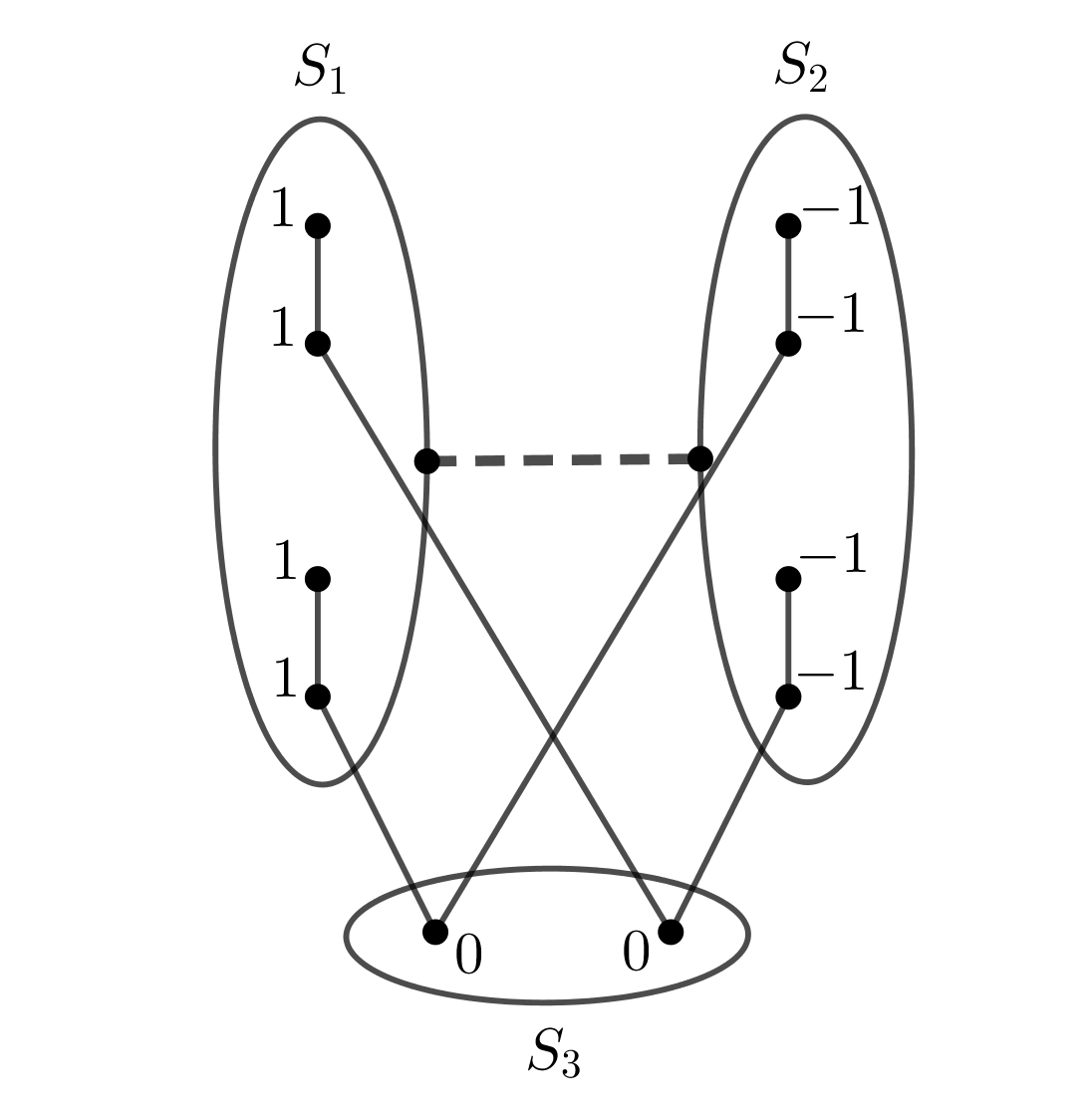}
\caption{Graph with $\lambda_{10} = -3$ and eigenvector with entries in $\{1,-1, 0\}$, where bold dashed line means that each vertex of $S_1$ is connected to all vertices in $S_2.$}
\label{fig:family1}
\end{figure}

\vspace{0.2cm}

Now, consider the $n$-vector given by
$$p^{\prime}_{i} =\left\lbrace \begin{array}{rl}
c_1,&\text{ for }v_i \in S_1,\\
-c_2,&\text{ for }v_i \in S_2,\\
0 ,&\text{ for }v_i \in S_3.
\end{array}\right. $$
The proof technique of Theorem \ref{thm:adjpartition} can be applied to characterize all
graphs with eigenvectors of the type $\{c_1, -c_2, 0\}^{n},$ where $c_1, c_2 \in \mathbb{R}_+^*.$

\begin{theorem}\label{Prop03-vec-pond-for-A} Let $G=(V,E)$ be a graph on $n$ vertices and let $p^{\prime} \in \{c_1, -c_2, 0\}^n$, for $c_1,c_2 \in \mathbb{R}_+^*$, be the vector associated to the partition $\{S_1, S_2, S_3\}$ of $V.$ Then, $p^{\prime}$ is an  eigenvector of $A$ to the eigenvalue $\lambda$ if and only if
\begin{eqnarray}
c_1 d_{v_i}(S_1)-c_2   d_{v_i}(S_1,S_2)&=&c_1   \lambda, \mbox{  for} \; v_i \in S_1, \label{Prop03-eqA}\\
c_2 d_{v_i}(S_2)-c_1   d_{v_i}(S_2,S_1)&=&c_2   \lambda, \mbox{  for}\; v_i \in S_2, \label{Prop03-eqB} \\
c_1  d_{v_i}(S_3,S_1)-c_2  d_{v_i}(S_3,S_2)&=&0, \mbox{  for}\; v_i \in S_3. \label{Prop03-eqC}
\end{eqnarray}
\end{theorem}

\begin{remark}\label{rem2} The Equations \eqref{Prop03-eqA}, \eqref{Prop03-eqB}, and \eqref{Prop03-eqC} can be interpreted in different ways, including for weighted digraphs. The graph of Figure \ref{fig:family2} satisfies the conditions of Theorem \ref{Prop03-vec-pond-for-A} for $c_1=1$ and $c_2 = -2$ since $\lambda_{3}=2.$
\end{remark}

\begin{figure}
\centering
\includegraphics[height = 8cm]{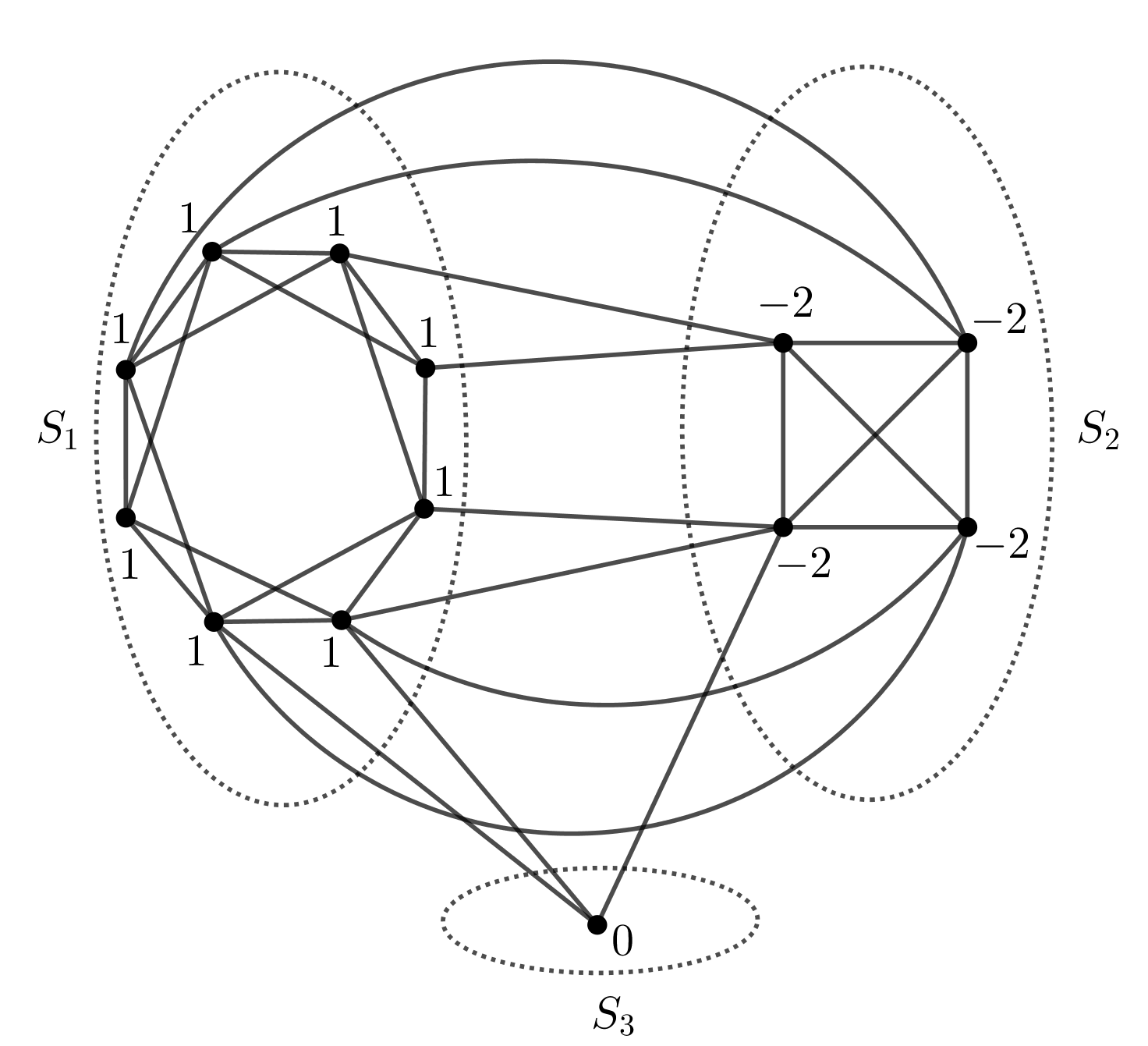}
\caption{Graph with $\lambda_{3} = 2$ and eigenvector with entries in $\{1,-2,0\}$.}
\label{fig:family2}
\end{figure}

\vspace{0.3cm}

The following theorem generalizes a result obtained by de Lima
and Alencar in \cite{AL21} (Theorem 4.1).

\vspace{0.3cm}

\begin{theorem}\label{Prop06-pond-part-vec-A} Let $G=(V,E)$ be a graph on $n$ vertices and let $p\in \{c_1,-c_2\}^n$, for $c_1,c_2 \in \mathbb{R}_+^*$, a vector associated with the partition  $\{S,\overline{S}\}$ of $V.$ Then
$p^{\prime}$ is an eigenvector to the eigenvalue $\lambda$ of $A$ if and only if
\begin{eqnarray}\label{Prop06-eqA}
\begin{array}{cccl}
c_1   d_{v_i}(S)-c_2    d_{v_i}(S,\overline{S})&=&c_1    \lambda,& \mbox{for } \; v_i \in S,\\
c_2   d_{v_i}(\overline{S})-c_1    d_{v_i}(\overline{S},S)&=&c_2    \lambda,& \mbox{for }\; v_i \in \overline{S}.\\\end{array}
\end{eqnarray}
\end{theorem}

\vspace{0.4cm}

The next theorem characterizes the graphs where $p \in R^{n}$ is an eigenvector to the signless Laplacian $Q.$

\begin{theorem}\label{thm:slpartition} Let $G=(V,E)$ be a graph on $n$ vertices and let $p \in \{1,-1,0\}^n,$  be the vector associated to the partition $\{S_1, S_2, S_3\}$ of $V.$ Then, $p$ is an eigenvector to an eigenvalue $q$ of $Q$ if and only if
\begin{eqnarray}
2d_{v_i}(S_1) +  d_{v_i}(S_1,S_3)&=& q, \mbox{  for  } \; v_i \in S_1, \nonumber \\
2d_{v_i}(S_2) +  d_{v_i}(S_2,S_3)&=&q, \mbox{  for  }\; v_i \in S_2, \nonumber \\
d_{v_i}(S_3,S_1)- d_{v_i}(S_3,S_2)&=&0, \mbox{  for  }\; v_i \in S_3. \nonumber
\end{eqnarray}
\end{theorem}

\begin{proof}
Let $G^{\prime}$ be the weighted graph obtained from $G$ by adding a loop $(v_i, v_i)$ to each vertex $v_i \in V(G)$ such that $w_{ii} = \sum_{v_j \sim v_i} w_{ij}.$ Write $d^{\prime}_{i}(X)$ and $d^{\prime}_{i}(X,Y)$ for the degrees of vertex $v_i$ in $G^{\prime}[X]$, and in $G^{\prime}[E(X,Y)]$, respectively. Note that $A^{\prime} = A(G^{\prime}) = Q(G).$ From Theorem \ref{thm:adjpartition} (with Remark \ref{rem1}), $p$ is an eigenvector to an eigenvalue $q$ of $Q$ if and only if
\begin{eqnarray}
d^{\prime}_{v_{i}}(S_1)- d^{\prime}_{v_{i}}(S_1,S_2)&=&  q, \mbox{  for } \; v_i \in S_1, \nonumber \\
d^{\prime}_{v_{i}}(S_2)- d^{\prime}_{v_{i}}(S_2,S_1)&=&  q, \mbox{  for }\; v_i \in S_2,\nonumber \\
d^{\prime}_{v_{i}}(S_3, S_1) - d^{\prime}_{v_{i}}(S_3,S_2)&=& 0 , \mbox{  for }\; v_i \in S_3. \nonumber
\end{eqnarray}
Since $d^{\prime}_{v_{i}}(S_j) = 2 d_{v_i}(S_j) + d_{v_i}(S_j, S_k) + d_{v_i}(S_j, S_t)$ and $d^{\prime}_{v_{i}}(S_j, S_k) = d_{v_i}(S_j, S_k)$ for every $v_{i} \in V(G^{\prime}), j \in \{1,2,3\}, k,t \in \{1,2,3\}\setminus \{j\}$ and $k \neq t$, we have

\begin{eqnarray}
2    d_{v_i}(S_1)+ d_{v_i}(S_1,S_3)&=&    q, \mbox{  for } \; v_i \in S_1, \nonumber \\
2    d_{v_i}(S_2)+ d_{v_i}(S_2,S_3)&=&    q, \mbox{  for } \; v_i \in S_2, \nonumber \\
  d_{v_i}(S_3,S_1)-    d_{v_i}(S_3,S_2)&=&0, \mbox{  for } \; v_i \in S_3. \nonumber
\end{eqnarray}

This completes the proof.
\end{proof}

\vspace{0.3cm}


A result similar to Theorem \ref{thm:slpartition} for the Laplacian can be
easily obtained from the results of \cite{caputo2019}. Using the ideas of Theorem \ref{Prop03-vec-pond-for-A} and \ref{Prop06-pond-part-vec-A}, we characterize all graphs with an eigenvector to $Q$ and $L$ of the type $p^{\prime} \in \{c_1,-c_2,0\}^n$, for $c_1,c_2 \in \mathbb{R}_+^*$.

\begin{corollary}\label{Prop04-vec-pond-for-QL} Let $G=(V,E)$ be a graph on $n$ vertices and let $p^{\prime} \in \{c_1,-c_2,0\}^n$, for $c_1, c_2 \in \mathbb{R}_{+}^{*}$, be a vector associated to the partition $\{S_1, S_2, S_3\}$ of $V.$ Then, $p^{\prime}$ is an eigenvector to an eigenvalue $q$ of $Q$ if and only if
\begin{eqnarray}
c_1    d_{v_i}(G)+c_1    d_{v_i}(S_1)-c_2   d_{v_i}(S_1, S_2) &=&c_1    q, \mbox{  for  } \; v_i \in S_1, \nonumber \\
c_2    d_{v_i}(G)+c_2    d_{v_i}(S_2)-c_1   d_{v_i}(S_2, S_1) &=&c_2    q, \mbox{  for  } \; v_i \in S_2, \nonumber \\
c_1    d_{v_i}(S_3, S_1)-c_2   d_{v_i}(S_3, S_2) &=& 0, \mbox{  for  } \; v_i \in S_3.\nonumber
\end{eqnarray}
\end{corollary}

\begin{corollary} Let $G=(V,E)$ be a graph on $n$ vertices and let $p\in \{c_1,-c_2\}^n$, for $c_1,c_2 \in \mathbb{R}_+^*$, a vector associated to the bipartition $\{S,\overline{S}\}$ of $V.$ Then, $p$ is an eigenvector to an
eigenvalue $q$ of $Q$ if and only if
\begin{eqnarray*}
 \begin{array}{cccl}
c_1   d_{v_i}(G)+ c_1 d_{v_i}(S)-c_2    d_{v_i}(S,\overline{S})&=&c_1    q,& \mbox{for } \; v_i \in S, \\
c_2   d_{v_i}(G)+ c_2   d_{v_i}(\overline{S})-c_1    d_{v_i}(\overline{S},S)&=&c_2    q,& \mbox{for }\; v_i \in \overline{S}.\\\end{array}
\end{eqnarray*}

\end{corollary}


\begin{corollary} Let $G=(V,E)$ be a graph on $n$ vertices and let $p^{\prime} \in \{c_1,-c_2,0\}^n$, for $c_1, c_2 \in \mathbb{R}_{+}^{*}$, be a vector associated to the partition $\{S_1, S_2, S_3\}$ of $V.$ Then, $p^{\prime}$ is an eigenvector to an eigenvalue $\mu$ of $L$ if and only if
\begin{eqnarray}
c_1    d_{v_i}(G)-c_1    d_{v_i}(S_1)+c_2   d_{v_i}(S_1, S_2) &=&c_1    \mu, \mbox{  for} \; v_i \in S_1, \nonumber \\
c_2    d_{v_i}(G)-c_2    d_{v_i}(S_2)+c_1   d_{v_i}(S_2, S_1) &=&c_2    \mu, \mbox{  for} \; v_i \in S_2, \nonumber \\
c_1    d_{v_i}(S_3, S_1)-c_2   d_{v_i}(S_3, S_2) &=& 0, \mbox{  for} \; v_i \in S_3.\nonumber
\end{eqnarray}
\end{corollary}

\begin{corollary} Let $G=(V,E)$ be a graph on $n$ vertices and let $p\in \{c_1,-c_2\}^n$, for $c_1,c_2 \in \mathbb{R}_+^*$, a  vector associated to the bipartition $\{S,\overline{S}\}$ of $V.$ Then, $p$ is an eigenvector to an eigenvalue $\mu$ of $L$ if and only if
\begin{eqnarray*}
\begin{array}{cccl}
(c_1+c_2)    d_{v_i}(S,\overline{S})&=&c_1    \mu,& \mbox{for } \; v_i \in S,\\
(c_1+c_2)    d_{v_i}(\overline{S},S)&=&c_2    \mu,& \mbox{for }\; v_i \in \overline{S}.\\\end{array}
\end{eqnarray*}

\end{corollary}

\textcolor{blue}{For each $i\,=\,1,\ldots, n,$ write $\lambda_{i}(M)$ to the $i$-th largest eigenvalue of a square matrix $M$ of order $n.$} We state the well-known Weyl's inequality used in Section 5. The equality case in Weyl's inequalities has been proved
by Wasin So in \cite{So94}.

\begin{theorem}[\cite{So94}]\label{weyltheo}
Let $A$ and $B$ be Hermitian matrices of order $n$, and let $1\leq i\leq n$ and $1 \leq j\leq n$. Then
\begin{eqnarray}
\lambda _i(A) + \lambda _j(B) \leq \lambda _{i+j-n}(A+B),\text{ if } i+j \geq n + 1, \label{in:weyl1} \\
\lambda _i(A) + \lambda _j(B) \geq \lambda _{i+j-1}(A+B),\text{ if } i+j \leq n + 1. \label{in:weyl2}
\end{eqnarray}
In either of these inequalities, equality holds if and only if there exists a nonzero $n$-vector that is an eigenvector to each of the three involved eigenvalues.
\end{theorem}

\section{$k$-cut of square matrices}

Write $\Sigma A$ for the sum of the entries of a matrix $A.$ We next introduce $k$-cut of square matrices. Let $A$ be an $n\times n$ matrix, and let $\mathcal{V}=\left\{  V_{1}%
,\ldots,V_{k}\right\}  $ be a partition of the set $\left[  n\right]
=\left\{  1,\ldots,n\right\}.$ 
Write $A_{p,q}$ for the submatrix of $A$ consisting of all entries $a_{i,j}$
such that $i\in V_{p}$ and $j\in V_{q}$. Now let%
\[
\textcolor{blue}{2} \; \mathrm{cut}\left(  A,\mathcal{V}\right)  =\Sigma A-\Sigma A_{1,1}-\cdots-\Sigma
A_{k,k}.
\]
It is not hard to see that if $A$ is the adjacency matrix of a graph $G$, then $\mathrm{cut}\left(  A,\mathcal{V}\right)$ is the size of the $k$-partite subgraph of $G$, whose partition is given by $\mathcal{V}$. 

\begin{theorem}
\label{mth}Let $A$ be a symmetric matrix of order $n$, and let $\mathcal{V}%
=\left\{  V_{1},\ldots,V_{k}\right\}  $ be a partition of $\left[  n\right]
.$ If $B$ is a\ diagonal matrix of order $n,$ then
\begin{equation}
\lambda_{1} \left(  B-A\right)  \geq\frac{\textcolor{blue}{2}k}{\left(  k-1\right)  n}\mathrm{cut}%
\left(  A,\mathcal{V}\right)  -\frac{1}{n}\Sigma A+\frac{1}{n}\Sigma
B.\label{mcin}%
\end{equation}

\end{theorem}

\begin{proof}
We use Rayleigh's principle to construct $k$ lower bounds on $\lambda_{1} \left(
B-A\right)  ,$ and take their average as a lower bound on $\lambda_{1} \left(
B-A\right)  $. Set for short $C=B-A.$

For each $i\in\left[  k\right]  ,$ define a vector $\mathbf{s}^{(i)}%
:=(s_{1}^{(i)},\ldots,s_{n}^{(i)})$ as%
\[
s_{j}^{(i)}:=\left\{
\begin{array}
[c]{ll}%
-k+1, & \text{if }j\in V_{i},\medskip\\
1, & \text{if }j\in\left[  n\right]  \backslash V_{i}.
\end{array}
\right.
\]
Write $\left\langle \mathbf{u},\mathbf{v}\right\rangle $ for the inner product
of the vectors $\mathbf{u}$ and $\mathbf{v,}$ and note that for each
$i\in\left[  k\right]  ,$ Rayleigh's principle implies that
\[
\lambda_{1} \left(  C\right)  \left\Vert \mathbf{s}^{(i)}\right\Vert ^{2}%
\geq\left\langle C\mathbf{s}^{(i)},\mathbf{s}^{(i)}\right\rangle .
\]
Hence, summing these inequalities for all $i\in\left[  k\right]  ,$ we get%
\begin{equation}
\lambda_{1}\left(  C\right)  \sum\limits_{i\in\left[  k\right]  }\left\Vert
\mathbf{s}^{(i)}\right\Vert ^{2}\geq\sum\limits_{i\in\left[  k\right]
}\left\langle C\mathbf{s}^{(i)},\mathbf{s}^{(i)}\right\rangle .\label{i1}%
\end{equation}
On the one hand, for $\sum_{i\in\left[  k\right]  }\left\Vert \mathbf{s}%
^{(i)}\right\Vert ^{2}$ we have
\begin{equation}
\sum\limits_{i\in\left[  k\right]  }\left\Vert \mathbf{s}^{(i)}\right\Vert
^{2}=\sum\limits_{i\in\left[  k\right]  }\left(  \left(  k-1\right)
^{2}\left\vert V_{i}\right\vert +n-\left\vert V_{i}\right\vert \right)
=k\left(  k-1\right)  n.\label{i2}%
\end{equation}
On the other hand, for every $i\in\left[  k\right]  ,$ we see that
\[
\left\langle C\mathbf{s}^{(i)},\mathbf{s}^{(i)}\right\rangle =\left(
k-1\right)  ^{2}\Sigma C_{i,i}-\textcolor{blue}{2}\left(  k-1\right)  \sum\limits_{j\in\left[
k\right]  \backslash\left\{  i\right\}  }\Sigma C_{i,j}+\sum\limits_{j\in
\left[  k\right]  \backslash\left\{  i\right\}  }\Sigma \textcolor{blue}{C_{j,j}}+\sum
\limits_{j\in\left[  k\right]  \backslash\left\{  i\right\}  ,m\in\left[
k\right]  \backslash\left\{  i,j\right\}  \text{ }}\Sigma C_{j,m}.
\]
Summing these inequalities for all $i\in\left[  k\right]  ,$ we get four terms on the right side:
\begin{align*}
\left(  k-1\right)  ^{2}\sum\limits_{i\in\left[  k\right]  }\Sigma C_{i,i} &
=\left(  k-1\right)  ^{2}\left(  \Sigma C-\textcolor{blue}{2}\mathrm{cut}\left(  C,\mathcal{V}%
\right)  \right)  ,\\
-2\left(  k-1\right)  \sum\limits_{i\in\left[  k\right]  }\sum\limits_{j\in
\left[  k\right]  \backslash\left\{  i\right\}  }\Sigma C_{i,j} &  =-\textcolor{blue}{4}\left(
k-1\right)  \mathrm{cut}\left(  C,\mathcal{V}\right)  ,\\
\sum\limits_{i\in\left[  k\right]  }\sum\limits_{j\in\left[  k\right]
\backslash\left\{  i\right\}  }\Sigma \textcolor{blue}{C_{j,j}} &  =\left(  k-1\right)  \left(
\Sigma C-\textcolor{blue}{2}\mathrm{cut}\left(  C,\mathcal{V}\right)  \right)  \\
\sum\limits_{i\in\left[  k\right]  }\text{ }\sum\limits_{
j\in\left[  k\right] \backslash\left\{  i\right\} , 
\, m\in\left[  k\right]  \backslash\left\{ i,i\right\}  \text{ }
}\Sigma C_{j,m} &  = \textcolor{blue}{2}\left(  k-2\right) \mathrm{cut}%
\left(  C,\mathcal{V}\right)  .
\end{align*}
Hence, for $\sum_{i\in\left[  k\right]  }\left\langle C\mathbf{s}%
^{(i)},\mathbf{s}^{(i)}\right\rangle $ we obtain
\begin{align*}
\sum\limits_{i\in\left[  k\right]  }\left\langle C\mathbf{s}^{(i)}%
,\mathbf{s}^{(i)}\right\rangle  &  =\left(  k-1\right)  ^{2}\left(  \Sigma
C-\textcolor{blue}{2}\mathrm{cut}\left(  C,\mathcal{V}\right)  \right)  -\textcolor{blue}{4}\left(  k-1\right)
\mathrm{cut}\left(  C,\mathcal{V}\right)  \\
&  +\left(  k-1\right)  \left(  \Sigma C-\textcolor{blue}{2}\mathrm{cut}\left(  C,\mathcal{V}%
\right)  \right)  +\textcolor{blue}{2}\left(  k-2\right)  \mathrm{cut}\left(  C,\mathcal{V}%
\right)  \\
&  =k\left(  k-1\right)  \left(  \Sigma C-\textcolor{blue}{2}\mathrm{cut}\left(  C,\mathcal{V}%
\right)  \right)  -\textcolor{blue}{2}k\mathrm{cut}\left(  C,\mathcal{V}\right)  \\
&  =k\left(  k-1\right)  \left(  \Sigma C-\frac{\textcolor{blue}{2}k}{k-1}\mathrm{cut}\left(
C,\mathcal{V}\right)  \right)  .
\end{align*}
Combining the last equality with (\ref{i1}) and (\ref{i2}), we get%
\[
\lambda_{1} \left(  C\right)  k\left(  k-1\right)  n\geq k\left(  k-1\right)
\left(  \Sigma C-\frac{\textcolor{blue}{2}k}{k-1}\mathrm{cut}\left(  C,\mathcal{V}\right)
\right).
\]
Now, noting that%
\[
\Sigma C=\Sigma B-\Sigma A
\]
and%
\[
\mathrm{cut}\left(  C,\mathcal{V}\right)  =-\mathrm{cut}\left(  A,\mathcal{V}%
\right)  ,
\]
we complete the proof of (\ref{mcin}).
\end{proof}

The main advantage of Theorem \ref{mth} is the freedom of choice of the matrix $B$, as illustrated in the following corollary.

\begin{corollary}\label{coro:bo}
Let $G$ be a graph of order $n$ and size $m$, and let $\lambda_{n}$ be its
smallest adjacency eigenvalue, $\mu_{1}$ be its largest Laplacian eigenvalue,
and $q_{n}$ be its smallest signless Laplacian eigenvalue. Then%
\begin{align}
\mathrm{mc}_{k}\left(  G\right)    & \leq\frac{k-1}{k}\left(  m-\frac
{n\lambda_{n}}{2}\right)  \text{;}\label{bo1}\\
\mathrm{mc}_{k}\left(  G\right)    & \leq\frac{\left(  k-1\right)  n}{2k}%
\mu_{1}\text{;}\label{bo2}\\
\mathrm{mc}_{k}\left(  G\right)    & \leq\frac{k-1}{k}\left(  2m-\frac
{nq_{n}}{2}\right)  .\label{bo3}%
\end{align}

\end{corollary}

\begin{proof}
Let $A$ be the adjacency matrix of $G$. Bound (\ref{bo1}) follows from Theorem
\ref{mth} by letting $B$ be the zero diagonal matrix and noting that
$\lambda_{1}\left(  -A\right)  =-\lambda_{n}.$

Likewise, bound (\ref{bo2}) follows by letting $B=D$.

Finally,  bound (\ref{bo3}) follows by letting $B=-D$, and noting that
$\lambda_{1}\left(  -D-A\right)  =-q_{n}$.
\end{proof}

Note that bound (\ref{bo2}) has been proved in \cite{vanDam2016}, and bound (\ref{bo1} has been proved in \cite{niki2016}. The three bounds are incomparable and each of them can be the best one for a given class of graphs.

\begin{remark}\label{rem3} Equality holds for \eqref{bo1}, \eqref{bo2} and \eqref{bo3}  when $\mathbf{s}^{(i)}$ is an eigenvector to $\lambda_{n}(G), q_{n}(G)$ and $\mu_{1}(G)$ for each $i = 1,\ldots,k.$ This fact will be useful to find extremal graphs to  those upper bounds, which are described in Section 5.
\end{remark}

\section{Graphs with $\mathbf{s}^{(i)}$ as eigenvectors}
\noindent For each $i\in \left[  k\right]  ,$ define a vector $\mathbf{s}^{(i)}%
:=(s_{1}^{(i)},\ldots,s_{n}^{(i)})$ as%
\[
s_{j}^{(i)}:=\left\{
\begin{array}
[c]{ll}%
-k+1, & \text{if }j\in V_{i},\medskip\\
1, & \text{if }j\in\left[  n\right]  \backslash V_{i}.
\end{array}
\right.
\]
In this section, for each $i \in \left[  k\right]$ we obtain necessary and sufficient conditions for  $\mathbf{s}^{(i)}$
to be an eigenvector to an eigenvalue of the matrices $A$, $L$ and $Q$ of a graph $G.$ Notice that when $\mathbf{s}^{(i)}$ is an eigenvector to the smallest eigenvalue of $A,$ then equality holds in \eqref{bo1} of Corollary  \ref{coro:bo}. Also, the next results show us that the fact of $\mathbf{s}^{(i)}$ being an eigenvector is intrinsically related to the degree structure of the graph.

\begin{theorem}\label{Teo-kcut-forA}
Let $G=(V,E)$ be a graph and let $\{S_1, \ldots, S_k\}$ be a $k$-partition of $V(G)$ for $k \geq 2.$ For each $i \in [k],$    $\mathbf{s}^{(i)}$ is an eigenvector to the eigenvalue  $\lambda$ if and only if
$$d_{v_l} (S_i)- d_{v_l} (S_i, S_j) = \lambda$$
for each $v_{l} \in S_{i}$ and for each $j \ne i$ such that $\, j \in [k].$
\end{theorem}

\begin{proof}
For $k=2$ the result holds by Theorem 4.1 of \cite{AL21}. For $k\geq 3$, let $\{S_1,\ldots, S_k\}$ be a $k$-partition of $V.$ For $i \neq j,$ consider the $n$-vector
$\mathbf{u}^{(ij)}$  given as

$$\mathbf{u}^{(ij)} = \frac{1}{k}(\mathbf{s}^{(i)}-\mathbf{s}^{(j)}).$$
Notice that
$$u^{(ij)}_r=\left\lbrace\begin{array}{ccl}
1&,& \mbox{ if } v_r \in S_i,\\
-1&,& \mbox{ if }  v_r \in  S_j,\\
0&,& \mbox{ if } v_r \in V \setminus \{S_i, S_j\}.
\end{array}\right.$$

Suppose that $(\lambda, \, \mathbf{s}^{(i)})$ and $(\lambda, \, \mathbf{s}^{(j)})$ are eigenpairs of $A$ for $i, \, j \in \{1,\ldots,k\}$ such that $i \ne j.$ In this case, $\mathbf{u}^{(ij)}$ is also an eigenvector to $\lambda.$ From Theorem \ref{thm:adjpartition}, there is a tripartition  $\{S_i, S_j, V\setminus \{S_i, S_j\} \}$ such that
equalities in \eqref{eq:adjpartition} hold, and the result follows.

Now, we prove the converse. By hypothesis, considering a vertex $v_{l} \in S_1,$ we have
\begin{eqnarray}
d_{v_l} (S_1)- d_{v_l} (S_1, S_2) &=& \lambda \nonumber \\
d_{v_l} (S_1)- d_{v_l} (S_1, S_3) &=& \lambda \nonumber \\
\vdots \nonumber \\
d_{v_l} (S_1)- d_{v_l} (S_1, S_k) &=& \lambda, \nonumber
\end{eqnarray}
which implies that $ d_{v_l} (S_1, S_2) = \cdots = d_{v_l} (S_1, S_k).$ Analogously, for each $v_{l} \in S_i$ for $i=2,\ldots,k,$ we obtain
$ d_{v_l} (S_i, S_{i+1}) = d_{v_l} (S_i, S_{i+2}) = \cdots = d_{v_l} (S_i, S_k).$ By Theorem \ref{thm:adjpartition}, the proof is complete.
\end{proof}

\begin{remark}\label{rem4} It is easy to see that Theorem \ref{Teo-kcut-forA} holds for weighted graphs with possible loops. This fact will be used to prove Corollaries \ref{Teo-kcut-forQ} and \ref{Teo-kcut-forL}.
\end{remark}

Figure \ref{Figure-TeoA} displays a graph $G$ and a partition of the vertex set into three subsets $S_1, S_2$ and $S_3$ such that
$$d_{v_l} (S_i)- d_{v_l} (S_i, S_j) = 0$$
for all vertices $v_{l} \in V(G)$ and $i, j \in \{1, 2, 3\}$ with $i\ne j.$
From Theorem \ref{Teo-kcut-forA},
$\lambda = 0$ is an eigenvalue of $A(G)$ of multiplicity at least $2$, with eigenvectors $$\; \mathbf{s^{(1)}} = [2,2,2,2,2,2,-1,\ldots,-1]^{t},$$  $$\;\mathbf{s^{(2)}} = [-1,-1,-1,-1,-1,-1, 2,2,2,2,2,2,-1,-1,-1,-1,-1]^{t}$$ and $\;\mathbf{s^{(3)}} =-\mathbf{s^{(1)}}-\mathbf{s^{(2)}}$. In a more general case, given a graph $G$ satisfying the hypothesis of Theorem \ref{Teo-kcut-forA}, the eigenvectors $\{\mathbf{s^{(1)}}, \ldots, \mathbf{s^{(k)}}\}$ show that the multiplicity of the correspondent eigenvalue is at least $(k-1)$.

\begin{figure}[!h]
\centering
\includegraphics[scale=0.2]{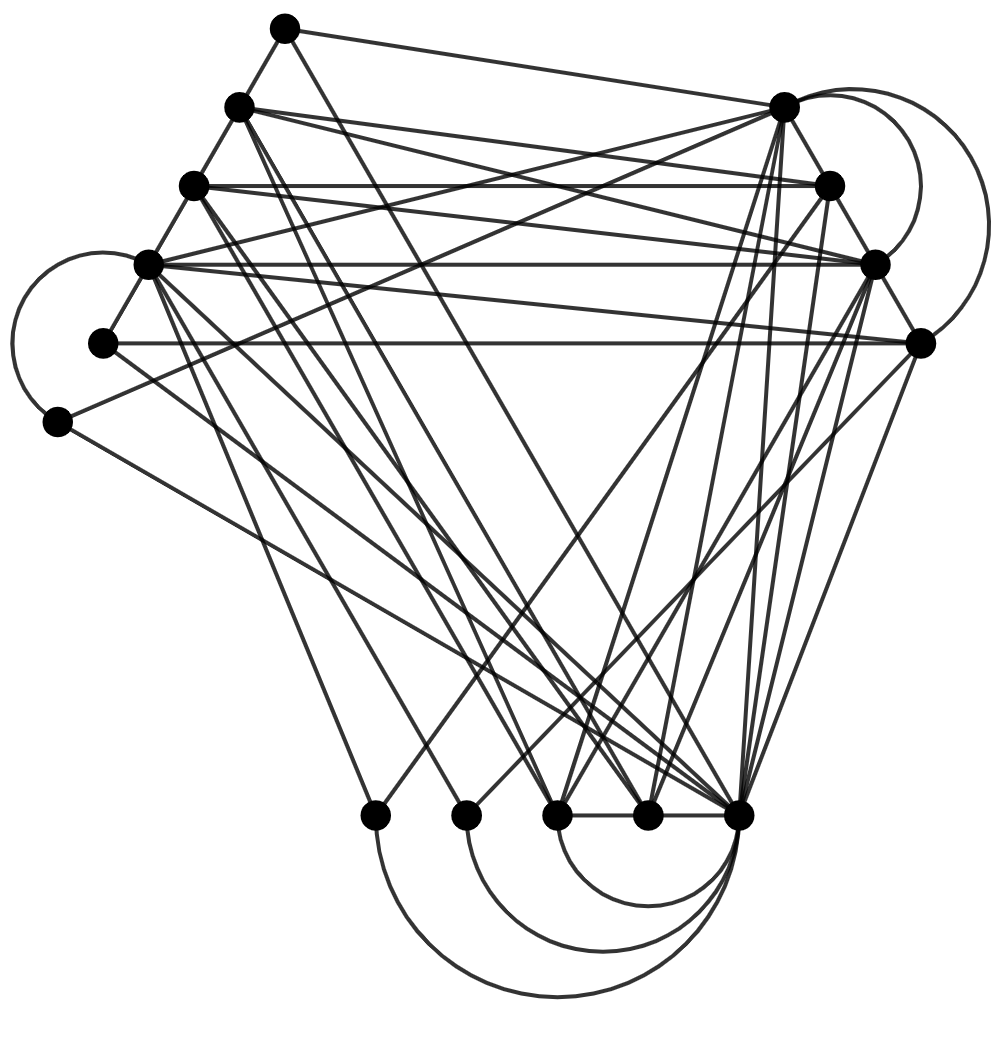}
\caption{Graph such that $ \mathbf{s}^{(1)}$ and  $\mathbf{s}^{(2)}$ belong to the eigenspace of $\lambda = 0$.}
\label{Figure-TeoA}
\end{figure}

%
%
 %
%
%


\begin{corollary}\label{Teo-kcut-forQ}
Let $G=(V,E)$ be a graph and let  $\{S_1, \ldots, S_k\}$ be a $k$-partition of $V(G)$ for $k \geq 2.$ Then, $(q , \mathbf{s}^{(i)})$ is an eigenpair of $Q$ if and only if
$$d_{v_l}(G)+d_{v_l}(S_i)-d_{v_l}(S_i,S_j)=q,\text{ }\forall i,j \in [k], i\neq j\text{ and } v_l \in S_i.$$
\end{corollary}
\begin{proof}
Let $G^{\prime}$ be the graph obtained from $G$ by adding a loop on each vertex $v \in V$ with a weight equal to the degree of the vertex $v$ on $G.$ Thus, $Q(G) = A(G^{\prime}).$
Applying Theorem \ref{Teo-kcut-forA} (with Remark \ref{rem4}) to the adjacency matrix of the graph $G^{\prime}$, $(q, \mathbf{s}^{(i)})$ is an eigenpair of $A^{\prime}(G)$ if and only if
$$d^{\prime}_{v_l} (S_i)- d^{\prime}_{v_l}(S_i,S_j) = q,\text{ }\forall i,j \in [k], i\neq j\text{ and } v_l \in S_i.$$
Since $d_{v_l}(S_i,S_j) = d^{\prime}_{v_l}(S_i,S_j)$ and $d^{\prime}_{v_l} (S_i) = d_{v_l}(S_i) + d_{v_l}(G),$ the result follows.
\end{proof}

Using the argument of Corollary \ref{Teo-kcut-forQ}, we obtain the following result for the Laplacian matrix of $G.$ Note that in this case, the Laplacian eigenvalue $\mu$ associated with eigenvector $\textbf{s}^{(i)}$ should be an integer multiple of $k.$

\begin{corollary}\label{Teo-kcut-forL} Let $G=(V,E)$ be a graph  and $\{S_1, \ldots, S_k\}$ be a $k$-partition of $V(G)$ such that $k \geq 2.$ For each $i=1,\ldots,k,$   $(\mu, \mathbf{s}^{(i)})$ is an eigenpair of $L$ if and only if the graph induced by $E(S_i,S_j)$ is $\frac{\mu}{k}$-regular.
\end{corollary}
\begin{proof}
Let $G^{\prime}$ be the graph obtained from $G$ by assigning edge weights -1, and also adding a self-loop to each vertex with weight equal to the vertex degree. In this case, we have $A(G^{\prime}) = L(G).$ Applying Theorem \ref{Teo-kcut-forA} (with Remark \ref{rem4}) we obtain that $(\mu, \mathbf{s}^{(i)})$ is an eigenpair of $A(G^{\prime})$ if and only if
$$d^{\prime}_{v_l} (S_i)- d^{\prime}_{v_l}(S_i,S_j) = \mu,\text{ }\forall i,j \in [k], i\neq j\text{ and } v_l \in S_i.$$

Notice that for $i,j,r \in [k]$

$$d_{v_l}(S_i,S_j)=d^{\prime}_{v_l}(S_i,S_j)=\mu-d^{\prime}_{v_l}(S_i)=d^{\prime}_{v_l}(S_i,S_r)=d_{v_l}(S_i,S_r).$$

\noindent Then, for all $i,j \in [k]$, $i\neq j$ and $v_l \in S_i$, we have

$$d_{v_l}(G) - d_{v_l} (S_i)-(- d_{v_l}(S_i,S_j)) = \mu,$$
$$(d_{v_l}(S_i)+\sum_{r\neq i}d_{v_l}(S_i,S_r)) - d_{v_l} (S_i)+ d_{v_l}(S_i,S_j) = \mu,$$
$$k \, d_{v_l}(S_i,S_j) = \mu, $$
$$d_{v_l}(S_i,S_j) = \frac{\mu}{k}.$$

\noindent Therefore, the graph induced by $E(S_i,S_j)$ is $\frac{\mu}{k}$-regular. Since the converse is straightforward, this completes the proof.
\end{proof}

From Corollary \ref{Teo-kcut-forL}, we can easily prove the following results.

\begin{corollary}
Let $G=(V,E)$ be a graph  and $\{S_1, \ldots, S_k\}$ be a $k$-partition of $V(G)$ such that $k \geq 2.$ For each $i \in \{1, \ldots,k-1\},$ if $(\mu, \mathbf{s}^{(i)}))$ is an eigenpair of $L$, then $|S_i|=|S_j|$ for each $j \in \{1, \ldots, k\}.$
\end{corollary}

\begin{corollary}\label{coro}
Let $G=(V,E)$ be a graph  and $\{S_1, \ldots, S_k\}$ be a $k$-partition of $V(G)$ such that $k \geq 2.$ For each $i \in \{1, \ldots,k-1\},$ if $(\mu, \mathbf{s}^{(i)}))$ is an eigenpair of $L$, then the graph induced by $\bigcup_{i\neq j} E(S_i,S_j)$ is $\frac{\mu (k-1)}{k}$-regular.
\end{corollary}

It is worth to note that, in the conditions of all previous results (Theorem \ref{Teo-kcut-forA}, Corollaries \ref{Teo-kcut-forQ} and \ref{Teo-kcut-forL}), for $i,j,r \in [k]$ and $j\neq i \neq r$, we have
$$d_{v_l}(S_i,S_j)=d_{v_l}(S_i,S_r), \mbox{ for each  } v_l\in S_i.$$

This equation can be used to construct a counterexample to the converse of Corollary \ref{coro}.

\section{Infinite families of extremal graphs}

\noindent An infinite family of extremal graphs for \eqref{bo1} were obtained at \cite{niki2016}. In this section, we construct new infinite families of extremal graphs for inequalities \eqref{bo1}, \eqref{bo2} and \eqref{bo3}. Figures \ref{Figure-ExampleA} and \ref{Figure-ExampleQ} display non-trivial graphs where $\mathbf{s}^{(i)}$ is an eigenvector, for $i = 1,2,3$, for the matrix $A$ and $Q$ respectively, but they are not for other matrices. 

\subsection{Extremal graphs for inequality  \eqref{bo1}}\

Next, Corollary \ref{coro:thm41} is a straightforward consequence of Theorem \ref{Teo-kcut-forA}, which we will use to build the extremal graphs.
\begin{corollary}
\label{coro:thm41}
Let $G=(V,E)$ be a graph and let $\{S_1, \ldots, S_k\}$ be a $k$-partition of $V(G)$ for $k \geq 3,$ and $H$ a proper subgraph of $G$ with a $k$-partition $\{S^{\prime}_1, \ldots, S^{\prime}_k\}$ of $V(H)$ such that $S^{\prime}_i \, \subset \,  S_i$, for $i =1,\ldots,k$. If both graphs, $G$ and $H$, satisfy the premises of Theorem \ref{Teo-kcut-forA} for the eigenvalues $\lambda$ and zero respectively, then $G^{\prime}=G-E(H)$ with the $k$-partition $\{S_1, \ldots, S_k\}$ of $V(G)$ also satisfies the premises of Theorem \ref{Teo-kcut-forA} for the eigenvalue $\lambda$.
\end{corollary}

\begin{proof}
Since both $G$ and $H$ satisfy the premises of Theorem \ref{Teo-kcut-forA} for the eigenvalues $\lambda$ and zero respectively, we have
$$d_{v_l}^H(S^{\prime}_i)-d_{v_l}^H(S^{\prime}_i,S^{\prime}_j)=0,$$
$$d_{v_l}^G(S_i)-d_{v_l}^G(S_i,S_j)=\lambda$$
for any $v_l \in S^{\prime}_i\subset S_i.$
Thus, for $G^{\prime}$
\begin{align*}
d_{v_l}^{G^{\prime}}(S_i)-d_{v_l}^{G^{\prime}}(S_i,S_j)&=(d_{v_l}^{G}(S_i)-d_{v_l}^{H}(S^{\prime}_i))-(d_{v_l}^G(S_i,S_j)-d_{v_l}^H(S^{\prime}_i,S^{\prime}_j)),\\
d_{v_l}^{G^{\prime}}(S_i)-d_{v_l}^{G^{\prime}}(S_i,S_j)&=(d_{v_l}^{G}(S_i)-d_{v_l}^G(S_i,S_j))-(d_{v_l}^{H}(S^{\prime}_i)-d_{v_l}^H(S^{\prime}_i,S^{\prime}_j)),\\
d_{v_l}^{G^{\prime}}(S_i)-d_{v_l}^{G^{\prime}}(S_i,S_j)&=\lambda.
\end{align*}
\noindent From Theorem \ref{Teo-kcut-forA}, the conclusion follows.
\end{proof}

Write $K_r$ for the complete graph on $r$ vertices, $C_r$ for the cycle on $r$ vertices with edge set $E(C_r).$ Now, we are able to build an infinite family of non-regular graphs attaining  equality in \eqref{bo1}.

\begin{proposition}
For $k\geq 3$ and $r\geq k+4,$ let $G^{\prime} = \overline{k(K_r-E(C_r))} - E(K_{k,k}).$ Then,
\begin{equation}\label{prop:Aextremal}
mc_{k}(G^{\prime}) = \left(\frac{k-1}{k} \right) \left(m-\frac{ n \lambda_n(G)}{2} \right).
\end{equation}
\end{proposition}
\begin{proof} For $k\geq 3$ and $r\geq k+4,$  let $G = \overline{k(K_r-E(C_r))}.$
We consider the vertex partition of $G$ into $k$ vertices sets $\{S_1,\ldots,S_k\}$ where $|S_{i}| = r$, $G[S_i] = C_r,$ and we label each vertex in $S_i$ as $V(S_{i}) = \{v_{1,i},\ldots, v_{r,i}\}$ for all $i=1,\ldots,k.$ It is not difficult to see that $\lambda_1(G) = r(k-1)+2$, $\lambda_{n-k+2}(G) = \lambda_{n}(G) = 2-r$ are the largest and the smallest eigenvalues of $G,$ respectively. Also, $\;2 \cos\left( \frac{2\pi j}{r}\right)$ are eigenvalues of $G$ of multiplicity $k$ for each $j=1, \ldots, r-1.$ We have
$$d_{v_l}(S_i)-d_{v_l}(S_i,S_j)= 2 - r$$
for any $v_{l} \in S_i$ of graph $G.$
From Theorem \ref{Teo-kcut-forA}, the vectors $\mathbf{s}%
^{\left(  j\right)  }$ associated with the $k$-partition $\left\{
S_{1},\ldots,S_{k}\right\}  $ are eigenvectors to the smallest eigenvalue.
%
%
Now, let $H^{\prime}=K_{k,k}$ be a complete bipartite graph that is a subgraph of $G$ such that $V(H^{\prime}) = (v_{1,1},\, v_{r,1}, \,\ldots, v_{1,i},\, v_{r,i},\,\ldots, v_{1,k},\, v_{r,k}).$ Now, considering the remaining $n-2k$ vertices of $G$, let $H = H^{\prime} \cup (n-2k)K_{1}.$
%
Note that
$$d^{H}_{v_l}(S_i)-d^{H}_{v_l}(S_i, S_j)= 0$$
for any $v_l \in S_i$ of graph $H.$ From Theorem \ref{Teo-kcut-forA}, each vector  $\mathbf{s}^{(i)}$ for $i=1,\ldots, k$ is an eigenvector to  eigenvalue $0$.
Corollary \ref{coro:thm41} implies that $G^{\prime}=G-E(H)$ has eigenvalue  $(2-r)$ with eigenvectors $\mathbf{s}^{(i)}$, for $i=1,\ldots, k$, associated to the $k$-partition $\{S_1,\ldots,S_k\}$ from $G$. Next, we need to prove that   $\lambda_{n-k+1}(G^{\prime}) > 2-r.$

Since $\lambda_{n-k+1}(G) = 2\cos ( \frac{2\pi j}{r}) \geq -2$ and $r \geq k+4,$
by Theorem \ref{weyltheo}, we get
\begin{eqnarray}
\lambda_{n-k+1}(G^{\prime}) &\geq& \lambda_{n-k+1}(G) + \lambda_{n}(-H), \nonumber \\
&=& \lambda_{n-k+1}(G) -  \lambda_{1}(H), \nonumber \\
& \geq & -2-k  \nonumber \\
& \geq & 2-r. \nonumber
\end{eqnarray}
Since $G$ is regular, $\left(  1,\ldots,1\right)^{T}$ is an  eigenvector to $\lambda_{1}\left(  G\right).$ Note that an eigenvector to $\lambda_{1}(H)$ cannot be an eigenvector to $G$ since its entries are equal to either 1 or 0, and it is not orthogonal to the all ones vector. By the condition for equality of the Theorem \ref{weyltheo}, we obtain that
$$\lambda_{n-k+1}(G^{\prime}) > 2-r.$$
Thus, $\mathbf{s}^{(i)}$ is an eigenvector to the smallest eigenvalue of $G^{\prime}.$ From  Remark \ref{rem3}, the result follows.
\end{proof}
Figure \ref{Figure-ExampleA} shows a graph $G=\overline{3(K_7-E(C_7))}$ and a subgraph $H=K_{3,3}$ in red.
\begin{figure}[ht]
\centering
\includegraphics[scale=0.20]{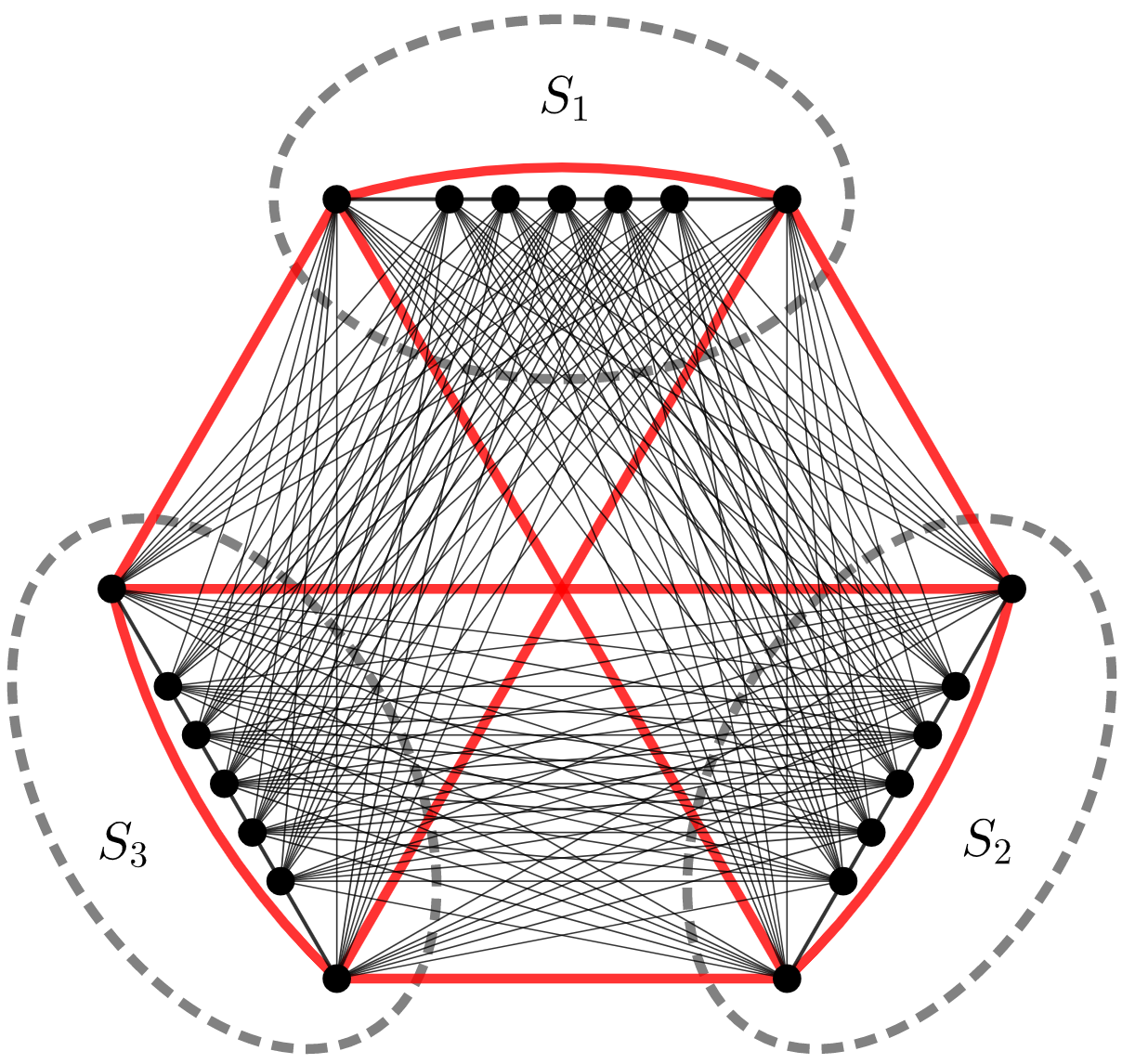}
\caption{Graphs  $G=\overline{3(K_7-E(C_7))}$ and $H=K_{3,3}$ in red.}
\label{Figure-ExampleA}
\end{figure}
The graph $G^{\prime}$ is a non-regular example where $\mathbf{s}^{(i)}$, for $i=1,\ldots,k$, are all eigenvectors to the smallest eigenvalue of $A$ that is equal to $-5.$ It is worth mentioning that $\mathbf{s}^{(i)}$ is not an eigenvector neither to $L$ nor to $Q.$

\subsection{Extremal graphs for inequality  \eqref{bo3}}\

Let $n_1$ be an even number and $n = 3n_1+4.$
Let $K_{n_1, n_1+2, n_1+2}$ be the complete 3-partite graph with parts $S_1, S_2$ and $S_3$ such that $|S_1| = n_1$ and $|S_2| = |S_3| = n_1 + 2.$ Consider labeling the vertex set of $H$ as $V(S_1) = \{v_{1,1}, \ldots, v_{1,n_1} \}$ and
$V(S_i) = \{v_{i,1}, v_{i,2}, \ldots, v_{i,n_{1}+2}\}$ for $i=2,3.$
Define the following edge sets:
\begin{eqnarray}
E_1 &=& \{(v_{1,1}, v_{1,2}), (v_{1,3}, v_{1,4}), \ldots, (v_{1,n_1-1}, v_{1,n_1}) \}, \nonumber\\
E_i &=& \{ (v_{i,1}, v_{i,2}), (v_{i,3}, v_{i,4}), \ldots, (v_{i,n_1+1}, v_{i,n_1+2})\}, \mbox{ for  } i = 2,3.  \nonumber
\end{eqnarray}
Also, consider the cycle $C$ with edge set given by
$$E(C) = \{ (v_{2,1}, v_{3,1}), (v_{3,1}, v_{2,2}), (v_{2,2}, v_{3,2})  \ldots, (v_{3,n_1+2}, v_{2,1})\}.$$
Now, let $H = (V_{H}, E_{H})$ be the graph with $V_{H} = V(K_{n_1, n_1+2, n_1+2})$ and edge set
$$E(H) = E(K_{n_1, n_1+2, n_1+2}) \cup E_1 \cup E_2 \cup E_3.$$
Note that $H$ can be written as the following join of regular graphs:
$H = \frac{n_1}{2} K_2 \vee \frac{(n_1+2)}{2} \, K_2 \vee \frac{(n_1+2)}{2} \, K_2$. Using Theorem 2.1 of \cite{Maria2010}, we obtain that $q_{n}(H) = n_1+2.$

Let $F\supset K_{n_1, n_1+2, n_1+2}$ be the supergraph with vertex set $V(F) = V(K_{n_1, n_1+2, n_1+2})$ and edge set $E(F) = E_1 \cup E(C).$
Figure \ref{Figure-ExampleQ} displays the subgraph $F$ and its edges in red, and the graph $H$ with its red and black edges when $n_1=4$.
\begin{figure}[!h]
\centering
\includegraphics[scale=0.2]{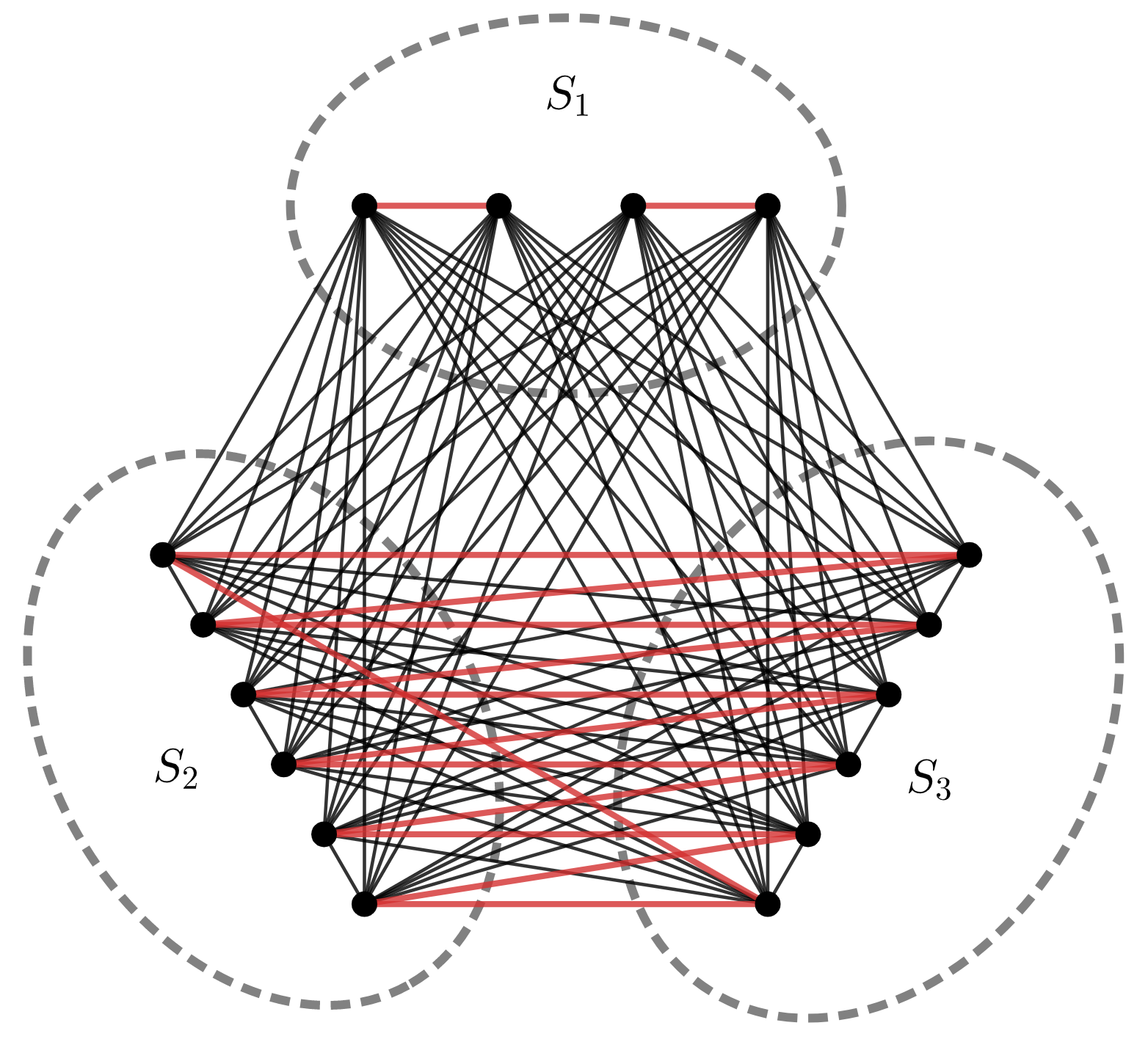}
\caption{Graphs $H$ and $F$ (induced by the red edges) when $n_1 = 4.$ }
\label{Figure-ExampleQ}
\end{figure}

Next, we prove that the third largest $Q$-eigenvalue of the complement of the subgraph induced by sets $S_2$ and $S_3$ minus the edges $E(C)$ is less than $n_1+4.$

\begin{proposition}\label{prop53}
Let $\,n_1 \geq 4\,$ be even and let $G = H[S_2, S_3] \setminus E(C)$ be a graph on $n$ vertices such that $n = 2n_1 + 4.$  Then,
$$q_{3}(\overline{G}) < n_1+4.$$
\end{proposition}
\begin{proof}
The graph $\overline{G}$ can be written as the union $\overline{\frac{n_1+2}{2}K_2} \, \cup \, \overline{\frac{n_1+2}{2}K_2} \, \cup \, E(C).$ Note that $\overline{G}$ is $(n_1+2)$-regular, so $q_1(\overline{G}) = n_1+2.$ For $S = S_2$ and $\overline{S} = S_3$, we have
\begin{eqnarray}
d_{v_i}(S) - d_{v_i}(S,\overline{S}) &=& n_1-2, \mbox{ for } v_i \in S, \nonumber \\
d_{v_i}(\overline{S}) - d_{v_i}(\overline{S}, S) &=& n_1-2, \mbox{ for } v_i \in \overline{S}. \nonumber
\end{eqnarray}
By Theorem 4.1 of \cite{AL21}, the partition vector $p \in \{-1,1\}$ is an eigenvector to the eigenvalue $n_1 - 2.$
Also, from Theorem \ref{weyltheo},
$$\lambda_3(\overline{G}) \leq \lambda_3\left(\overline{\tfrac{n_1+2}{2} K_2} \, \cup \, \overline{\tfrac{n_1+2}{2}K_2}\right) + \lambda_1(C_{2n_1+4}) = 2. $$
Since $\lambda_1(C_{2n_1+4})$ and $\lambda_3(\overline{G})$ do not share the same eigenvector $(1, \ldots, 1)^{T}$, then by the condition for equality of Theorem \ref{weyltheo}, $$\lambda_3(\overline{G}) <2.$$ For $r$-regular graphs, it is well-known that $q_i(G) = \lambda_{i}(G) + r.$ So, $q_1(\overline{G}) = 2n_1+4, q_2(\overline{G}) = 2n_1$ and $q_3(\overline{G}) < n_1+4.$ The proof is complete.
\end{proof}

\begin{proposition}
Let $\,n_1 \geq 4\,$ be even. Let $G=(V,E)$ be a  graph on $\,n = 3n_1 + 4\,$ vertices and $m$ edges such that $V(G) = V(H)$ and $E(G)=E(H) \setminus E(F).$ Then,
\begin{equation}
mc_{k}(G) = \left(\frac{k-1}{k} \right) \left( 2m-\frac{ n q_n(G)}{2} \right).
\end{equation}
\end{proposition}
\begin{proof}
A computational check shows that the result holds for $G$ when $n_1 = 4.$  Suppose that $n_1 \geq 6.$ From Corollary \ref{Teo-kcut-forQ}, $n_1+2$ is an eigenvalue for $Q(G)$ with multiplicity at least $2$, with eigenvectors $\mathbf{s}^{(j)}$ for $j=1, 2, 3.$ In order to obtain that $G$ is extremal to \eqref{bo3}, we need to prove that $q_{n}(G) = n_1+2.$

Observe that $\overline{G} = K_{n_1} \cup \hat{G},$ where $\hat{G} = \overline{H[S_2, S_3] \setminus E(C)}.$ From Proposition \ref{prop53},
$q_1(\overline{G}) = q_1(K_n) = 2n-2,$ $q_2(\overline{G}) = 2n_1,$ and $q_3(\overline{G}) = 2n_1-2 \geq n_1+4 > q_3(\hat{G}),$ for $n_1 \geq 6.$
Applying Theorem \ref{weyltheo} to $Q(G) = Q(K_n) - Q(\overline{G}),$ we have
\begin{eqnarray}
q_{n-2}(G) &\geq& q_{n}(K_{n}) - q_{3}(\overline{G})  \nonumber \\
&=& 3n_1+4-2-(2n_1-2) \nonumber \\
&=& n_1+4. \nonumber
\end{eqnarray}
So, $q_{n-1}(G) = q_{n}(G) = n_1+2,$ and the result follows by Remark \ref{rem3}.
\end{proof}

\subsection{Extremal graphs for inequality  \eqref{bo2}}

\vspace{0.3cm}
Constructing an infinite family of extremal graphs for inequality \eqref{bo2}, in the Laplacian case, is not  difficult if we observe the result of Corollary \ref{Teo-kcut-forL}.

Let $r\geq 2$ and let  $G=K_{r,\ldots, r}$ be the complete multipartite graph with $k \geq 3$ parts each of them of size $r.$ Notice that $\mu_1(G) = n.$ The graph induced by $E(S_{i}, S_{j})$ is $r$-regular, which implies from Corollary \ref{Teo-kcut-forL} that $\mathbf{s}^{(i)}$ is an eigenvector to $\mu_1(G).$ Let $G^{\prime}$ be the graph obtained from $G$ by adding an  edge. We get that $\mu_1(G^{\prime}) = n$ and the graph induced by $E(S_{i}, S_{j})$ in $G^{\prime}$ is still  $r$-regular, which implies from Corollary \ref{Teo-kcut-forL} that $\mathbf{s}^{(i)}$ is also an eigenvector to $\mu_1(G^{\prime}).$ Since $\mathbf{s}^{(i)}$ is an eigenvector to the largest Laplacian eigenvalue, then equality holds in \eqref{bo2} for both $G$ and $G^{\prime}$. This procedure can be applied many times to obtain non-isomorphic and non-regular graphs that force equality in \eqref{bo2}.

In general, suppose that $G$ is a graph such that $\{S_1,\ldots,S_k\}$ is a $k$-partition of its vertex set $V(G),$ and suppose that $(\mu_{1}(G), \mathbf{s}^{(i)})$ is an eigenpair of $L$ for each $i=1,\ldots,k.$ Now, let $G^{\prime}$ be obtained from $G$ by removing edges in $E(S_i)$, for any $i \in \{1,\ldots, k\}$, in any way. From Corollary \ref{Teo-kcut-forL},  $(\mu_{1}(G),\mathbf{s}^{(i)})$ is also an eigenpair of $L$ in $G^{\prime}$ since  by interlacing  $\mu_1(G')=\mu_1(G)$. Thus, $G$ and $G^{\prime}$ force equality in \eqref{bo2}.

\begin{remark}\label{rem5} It is worth mentioning that all results are also valid for weighted graphs with possible loops.
\end{remark}

\vspace{0.5cm}

\noindent \textbf{Acknowledgements: } The authors would like to thank the anonymous referee for his/her valuable comments/suggestions that helped us to improve the paper. The research of Leonardo de Lima was partly funded by the CNPq grant number 315739/2021--5.

\end{document}